\documentclass[journal]{IEEEtran}

\IEEEoverridecommandlockouts                              

\usepackage{changes}
\usepackage{authblk}
\usepackage{amsfonts}
\usepackage{amsthm}
\usepackage{amsmath,amssymb}
\usepackage{bm}
\usepackage{hyperref,cite}
\usepackage{cleveref}
\usepackage{color}
\usepackage{algorithm}
\usepackage{algorithmic}
\let\labelindent\relax
\usepackage{enumitem}
\usepackage{float}
\usepackage{graphicx}
\usepackage{subfigure}
\usepackage{booktabs}
\usepackage{threeparttable}
\usepackage{hhline}
\usepackage{scalerel}
\usepackage{ulem}
\usepackage{tikz}
\usepackage{multirow}
\usepackage{cancel}

\newtheorem{theorem}{Theorem}

\newtheorem{remark}{Remark}
\newtheorem{lemma}{Lemma}

\newcommand{\diag}{\mathop{\rm diag}\nolimits}

\newcommand{\rr}{{\mathbb R}}
\newcommand{\ba}[1]{\begin{array}{#1}}
\newcommand{\ea}{\end{array}}

\newcommand{\mr}[1]{\mathrm{#1}}
\newcommand*{\pdot}{\mathbin{\scalerel*{\boldsymbol\odot}{\circ}}}

\begin{document}
\title{\LARGE \bf
A Time-certified Predictor-corrector IPM Algorithm for Box-QP}
\author{Liang Wu$^{1,2}$, Yunhong Che$^{2}$, Richard D. Braatz$^{2}$, Jan Drgona$^{1}$%
\thanks{$^1$Johns Hopkins University, MD 21218, USA. $^2$Massachusetts Institute of Technology, MA 02139, USA. Corresponding author: 
Liang Wu with \tt\small liangwu@mit.edu.
}
}

\maketitle

\begin{abstract}
Minimizing both the \textit{worst-case} and \textit{average} execution times of optimization algorithms is equally critical in real-time optimization-based control applications such as model predictive control (MPC). Most MPC solvers have to trade off between certified \textit{worst-case} and practical \textit{average} execution times. For example, our previous work \cite{wu2025direct} proposed a full-Newton path-following interior-point method (IPM) with data-independent, simple-calculated, and \textit{exact} $O(\sqrt{n})$ iteration complexity, but not as efficient as the heuristic Mehrotra’s predictor–corrector IPM algorithm (which sacrifices global convergence). This letter proposes a new predictor–corrector IPM algorithm that preserves the same certified $O(\sqrt{n})$ iteration complexity while achieving a $5\times$ speedup over \cite{wu2025direct}. Numerical experiments and codes that validate these results are provided.
\end{abstract}


\section{Introduction}
This paper considers a scaled box-constrained quadratic program (Box-QP) with time-varying data $(H(t),h(t))$ as follows,
\begin{equation}\label{eqn_Box_QP}
    \begin{aligned}
        \min_{z\in\rr^n}&~ \frac{1}{2} z^\top H(t) z + z^\top h(t)\\
        \text{s.t.}&~  -\mathbf{1}_n\leq z\leq  \mathbf{1}_n,
    \end{aligned}
\end{equation}
where $H(t)\in\rr^{n\times n}$ is symmetric positive semi-definite. Without loss of generality, we assume that the box constraints are scaled to $[ -\mathbf{1}_n, \mathbf{1}_n]$. 

Time-varying Box-QP \eqref{eqn_Box_QP} often arises from real-time model predictive control (MPC) problems. For example, input-constrained MPC \cite{wu2025direct}, $\ell_1$-penalty soft-constrained MPC \cite{wu2025n3}, and model-penalized MPC \cite{saraf2017fast} can be formulated as a Box-QP \eqref{eqn_Box_QP}. While a shorter average computation time for MPC is desirable, the most critical factor is determining the worst-case computation time. Because the MPC solver has to return the optimal solution before the next feedback sampling time, referred to as the \textbf{execution time certificate}.  Recently, the \textbf{execution time certificate} of MPC (reduced to certifying the worst-case number of iterations if each iteration requires the same fixed number of floating-point operations) has attracted significant scholarly interest and remains a vibrant research area \cite{ richter2011computational, cimini2017exact, arnstrom2021unifying, okawa2021linear, kawaguchi2023bounded, wu2025direct, wu2025n3, wu2025eiqp}. A breakthrough is made by our previous work \cite{wu2025direct}, which proposed \textit{data-independent, simple-calculated}, and \textit{exact (not worst-case)} iteration complexity: $\left\lceil\frac{\log(\frac{2n}{\epsilon})}{-2\log(\frac{\sqrt{2n}}{\sqrt{2n}+\sqrt{2}-1})}\right\rceil + 1$ ($O(\sqrt{n})$-order) for Box-QP \eqref{eqn_Box_QP}. Then, it was applied to certifying the execution time of nonlinear MPC via the real-time-iteration scheme \cite{wu2024execution} and Koopman operator \cite{wu2024time}, respectively. Furthermore, \cite{wu2025eiqp} proposed exact iteration complexity: $\left\lceil\frac{\log(\frac{n+1}{\epsilon})}{-\log(1-\frac{0.414213}{\sqrt{n+1}})}\right\rceil$ for general convex QP. The algorithms in \cite{wu2025direct,wu2025eiqp} solve a linear system of equations (with $O(n^3)$) at each iteration, thus resulting $O(n^{3.5})$ time complexity. In \cite{wu2025n3}, by replacing the solving of linear equations with multiple rank-1 updates, for the first time proposed an \textit{implementable} $O(n^3)$-time-complexity QP algorithm (with $\mathcal{N}_\mathrm{iter}= \!\left\lceil\frac{\log\!\left(\frac{2n+\alpha\sqrt{2n}}{\epsilon}\right)}{-\log\!\left(1-\frac{\beta}{\sqrt{2n}}\right)}\right\rceil$ and $\mathcal{N}_{\mathrm{rank}-1}\leq \! \left\lceil\frac{4\eta(\mathcal{N}_\mathrm{iter}-1)\sqrt{n}}{\left(1-\eta\right)\log(1+\delta)}\right\rceil$). 

However, those time-certified interior-point-method (IPM) algorithms proposed in \cite{wu2025direct,wu2025eiqp,wu2025n3} are still not practically competitive: their certified computation times are typically longer than those of state-of-the-art QP solvers, which, in contrast, do not provide an execution time certificate. For example, the \textbf{heuristic} Mehrotra's predictor-corrector IPM algorithm \cite{mehrotra1992implementation}, renowned for its computational efficiency (empirically exhibiting $O(\log n)$ iteration complexity), has become the foundation of most IPM-based optimization software, yet its global convergence and theoretical iteration complexity bound remain unknown, and it even may diverge in some examples \cite[see p.\ 411]{nocedal2006numerical}, \cite{cartis2009some}. This naturally raises the question:
\begin{center}
\textit{Can we design a practically efficient predictor–corrector IPM-based Box-QP algorithm that also achieves the best-known certified iteration complexity of $O(\sqrt{n})$?}
\end{center}

\subsection{Contributions}
This paper gives a positive answer by proposing a predictor-corrector IPM algorithm with a certified, data-independent, and easily computable worst-case iteration bound, $N_{\max}=\left\lceil \frac{\log(\frac{2n}{\epsilon})}{-2\log\left(1-\frac{0.2348}{\sqrt{2n}}\right)}\right\rceil$ ($O(\sqrt{n})$-order), while empirically exhibiting much faster iteration complexity (such as $O(\log n)$ or $O(n^{0.25})$).

Although the proof framework is an extension of the predictor-corrector IPM framework from linear programs \cite{mizuno1993adaptive} to Box-QP \eqref{eqn_Box_QP}, the proof is more complicated than the original linear programming case, as the direction vectors are not orthogonal in the Box-QP case. More importantly, our proposed predictor–corrector IPM algorithm is \textbf{implementable} (code available at \url{https://github.com/liangwu2019/PC-BoxQP}), since a Box-QP admits cost-free initialization \cite{wu2025direct}, whereas the algorithm in \cite{mizuno1993adaptive} is not, as it assumes the availability of a strictly feasible initial point (usually requires solving another LP, known as the \textit{phase I} stage \cite[Sec. 11.4]{boyd2004convex}); consequently, no numerical experiments are reported in \cite{mizuno1993adaptive}.

Compared with our previous time-certified Box-QP algorithm \cite{wu2025direct}, the proposed algorithm demonstrates a $5\times$ speedup.

\section{Feasible Predictor-Corrector IPM Algorithm}
According to \cite[Ch 5]{boyd2004convex}, the Karush–Kuhn–Tucker (KKT) condition of Box-QP \eqref{eqn_Box_QP} is the following nonlinear equations,
\begin{subequations}\label{eqn_KKT}
\begin{align}
    H(t)z + h(t) + \gamma - \theta = 0,\label{eqn_KKT_a}\\
    z + \phi - \mathbf{1}_n=0,\label{eqn_KKT_b}\\
    z - \psi + \mathbf{1}_n=0,\label{eqn_KKT_c}\\
    (\gamma,\theta,\phi,\psi)\geq0,\label{eqn_KKT_d}\\
    \gamma \pdot \phi = 0,\label{eqn_KKT_e}\\
    \theta \pdot \psi = 0,\label{eqn_KKT_f}
\end{align}
\end{subequations}
where $\gamma,\theta$ are the Lagrangian variables of the lower and upper bound, respectively, and $\phi,\psi$ are the slack variables of the lower and upper bound, respectively. $\pdot$ represents the Hadamard product, i.e., $\gamma\pdot \phi = \mathrm{col}(\gamma_1\phi_1,\gamma_2\phi_2,\cdots{},\gamma_n\phi_n)$.

Path-following primal–dual IPMs are categorized into two types: \textit{feasible} and \textit{infeasible}, distinguished by whether the initial point satisfies Eqns. \eqref{eqn_KKT_a}–\eqref{eqn_KKT_d}. For the complementarity constraints \eqref{eqn_KKT_e}–\eqref{eqn_KKT_f}, \textit{feasible} path-following IPMs require the initial point to lie in a narrow neighborhood. To demonstrate this, let us denote the feasible region by $\mathcal{F}$, i.e.,
\begin{equation}
     \mathcal{F}=\{(z,\gamma,\theta,\phi,\psi):\eqref{eqn_KKT_a}\mathrm{-}\eqref{eqn_KKT_c},(\gamma,\theta,\phi,\psi)\geq0\}
\end{equation}
and the set of strictly feasible points by
\begin{equation}
    \mathcal{F}^+\triangleq\{(z,\gamma,\theta,\phi,\psi):\eqref{eqn_KKT_a}\mathrm{-}\eqref{eqn_KKT_c},(\gamma,\theta,\phi,\psi)>0\}
\end{equation}
We also consider the neighborhood
\begin{equation}
    \mathcal{N}(\beta)\triangleq\left\{(z,\gamma,\theta,\phi,\psi)\in\mathcal{F}^+: \left\|\left[\begin{array}{c}
         \gamma\pdot \phi  \\
         \theta\pdot \psi 
    \end{array}\right]-\mu \mathbf{1}_{2n} \right\|_2\leq\beta\mu\right\}
\end{equation}
where the duality measure $\mu\triangleq\frac{\gamma^\top \phi+\theta^\top\psi}{2n}$ and $\beta\in[0,1]$. \textit{Feasible} path-following IPMs require the initial point:
\begin{equation}\label{eqn_initial_point_requirement}
    (z^0,\gamma^0,\theta^0,\phi^0,\psi^0)\in\mathcal{N}(\beta),
\end{equation}
and computing such a point is typically expensive for general strictly convex QPs.

 \subsection{Cost-free initialization for Feasible IPMs}
Inspired by our previous work \cite{wu2025direct}, which was the first to point out that Box-QP supports cost-free initialization for feasible IPMs, this letter proposes the following initialization to ensure $ (z^0,\gamma^0,\theta^0,\phi^0,\psi^0)\in\mathcal{N}(\beta)$,
\begin{remark}\label{remark_initialization}
For $h=0$, the optimal solution of Box-QP \eqref{eqn_Box_QP} is $z^*=0$. For $h\neq0$,  first scale the objective as 
\[
\min_z \tfrac{1}{2} z^\top (2\lambda H) z + z^\top (2\lambda h)
\]
which does not affect the optimal solution 
and can ensure the initial point lies in $\mathcal{N}(\beta)$ if $\lambda\leftarrow\frac{\beta}{\sqrt{2}\|h\|_2}$. Then \eqref{eqn_KKT_a} is replaced by 
\[
2\lambda H z+2\lambda h+\gamma-\theta=0
\]
the initialization strategy for Box-QP \eqref{eqn_Box_QP} 
\begin{equation}\label{eqn_initialization_stragegy}
z^0 = 0,~\gamma^0 =\mathbf{1}_n - \lambda h,~\theta^0 =\mathbf{1}_n + \lambda h,~\phi^0 = \mathbf{1}_n, ~\psi^0 = \mathbf{1}_n,   
\end{equation}
which clearly places this initial point in $\mathcal{N}(\beta)$ by its definition in Eqn. \eqref{eqn_initial_point_requirement} (for example, $\left\|\left[\begin{array}{c}
         \gamma^0\pdot \phi^0  \\
         \theta^0\pdot \psi^0 
    \end{array}\right]-\mu \mathbf{1}_{2n} \right\|_2=\beta\mu$, where $\mu=1$). In particular, this letter chooses $\beta=\frac{1}{4}$, then $\lambda=\frac{1}{4\sqrt{2}\|h\|_2}$.
\end{remark}

\subsection{Algorithm descriptions}
For simplicity, we introduce
\[
v\triangleq\mathrm{col}(\gamma,\theta)\in\rr^{2n},~ s\triangleq\mathrm{col}(\phi,\psi)\in\rr^{2n}.
\]
According to Remark \ref{remark_initialization}, we have $(z,v,s)\in\mathcal{N}(\beta)$. Then, all the search directions
$(\Delta z,\Delta v,\Delta s)$
(for both predictor and corrector steps) are obtained as solutions of the following system of linear equations:
\begin{subequations}\label{eqn_Newton}
    \begin{align}
    (2\lambda H)\Delta z+ \Omega \Delta v &=0\label{eqn_Newton_a}\\
    \Omega^\top \Delta z + \Delta s & =0\label{eqn_Newton_b}\\
    s\pdot \Delta v + v\pdot \Delta s &=  \sigma \mu \mathbf{1}_{2n} - v\pdot s\label{eqn_Newton_c}
    \end{align}
\end{subequations}
where $\Omega=[I,-I] \in\mathbb{R}^{n \times 2n}$, $\sigma$ is chosen $0$ in predictor steps and $1$ in correctors steps, respectively, and $\mu\triangleq\frac{v^\top s}{2n}$ denotes the duality measure.
\begin{remark}\label{remark_positive}
Eqns. \eqref{eqn_Newton_a} and \eqref{eqn_Newton_b} imply that
\[
\Delta v^{\!\top\!}
\Delta s=\Delta v^{\!\top\!} (-\Omega^{\!\top\!}\Delta z)=\Delta z^{\!\top\!} (2\lambda H)\Delta z\geq0,
\]
which is critical in the following iteration complexity analysis. Note that in \cite{mizuno1993adaptive}, $\Delta v^{\!\top\!}
\Delta s=0$, thus making our analysis different and more complicated than the linear program case.
\end{remark}
By letting
\begin{equation}\label{eqn_Delta_gamma_theta_phi_psi}
    \begin{aligned}
        &\Delta \gamma= \sigma\mu\frac{1}{\phi}- \gamma + \frac{\gamma}{\phi}\Delta z,~\Delta \theta= \sigma\mu\frac{1}{\psi}-\theta-\frac{\theta}{\psi}\Delta z,\\
        &\Delta\phi = - \Delta z,~\Delta\psi = \Delta z,    
    \end{aligned}
\end{equation}
Eqn. \eqref{eqn_Newton} can be reduced into a more compact system of linear equations,
\begin{equation}{\label{eqn_compact_linsys}}
 \Big(2\lambda H+\diag\!\Big(\frac{\gamma}{\phi}\Big) + \diag\!\Big(\frac{\theta}{\psi}\Big) \!\Big) \Delta z= \sigma\mu\left(\frac{1}{\phi} - \frac{1}{\psi}\right) + \gamma - \theta.   
\end{equation}
The proposed feasible adaptive-step predictor-corrector IPM algorithm for Box-QP \eqref{eqn_Box_QP} is first described in Algorithm \ref{alg_PC_IPM}. In the next Subsection, we prove that Algorithm \ref{alg_PC_IPM} converges to the $\epsilon$-optimal solution ($v^\top s\leq\epsilon$) in the worst-case number of iterations
\begin{equation}
    N_{\max}=\left\lceil \frac{\log(\frac{2n}{\epsilon})}{-2\log\left(1-\frac{0.2348}{\sqrt{2n}}\right)}\right\rceil.
\end{equation}
which is an $O(\sqrt{n})$-order iteration complexity. In practice, Algorithm \ref{alg_PC_IPM} exhibits $O(n^{0.25})$ or $O(\log n)$-order iteration complexity due to the conservativeness of our proof.

\begin{algorithm}
    \caption{Time-certified predictor-corrector IPM for Box-QP \eqref{eqn_Box_QP}} \label{alg_PC_IPM}
    \textbf{Input}: Given a strictly feasible initial point  $(z^0,v^0,s^0)\in \mathcal{N}(1/4)$ from Remark \ref{remark_initialization} and a desired optimal level $\epsilon$. Then the worst-case iteration bound is $N_{\max}=\left\lceil \frac{\log(\frac{2n}{\epsilon})}{-2\log\left(1-\frac{0.2348}{\sqrt{2n}}\right)}\right\rceil$.
    \vspace*{.1cm}\hrule\vspace*{.1cm}

    \textbf{for} $k=0,1, 2,\cdots{},N_{\max}-1$ \textbf{do}
    \begin{enumerate}[label*=\arabic*., ref=\theenumi{}]
        \item if $(v^k)^\top s^k\leq \epsilon$, then break;
        \item Compute the predictor direction $(\Delta z_p,\Delta v_p,\Delta s_p)$ by solving Eqn. \eqref{eqn_Newton} with $(z,s,v)=(z^k,v^k,s^k)$, $\sigma\leftarrow0$, and $\mu\leftarrow\mu^k=\frac{(v^k)^\top s^k}{2n}$ (involving Eqns. \eqref{eqn_compact_linsys} and \eqref{eqn_Delta_gamma_theta_phi_psi});
        \item $\Delta \mu_{p}\leftarrow\frac{(\Delta v_p)^\top \Delta s_p}{2n}$;
        \item $\alpha^k\leftarrow \min\left(\frac{1}{2},\sqrt{\frac{\mu^{k}}{8\|\Delta v_p \pdot \Delta s_p -\Delta \mu_{p}\mathbf{1}_{2n}\|}}\right)$;
        \item $\hat{z}^k\leftarrow z^k+\alpha^k\Delta z_p,~\hat{v}^k\leftarrow v^k+\alpha^k\Delta v_p,~\hat{s}^k\leftarrow s^k+\alpha^k\Delta s_p$;
        \item Compute the corrector direction $(\Delta z_c,\Delta v_c,\Delta s_c)$ by solving Eqn. \eqref{eqn_Newton} with $(z,v,s)=(\hat{z}^k,\hat{v}^k,\hat{s}^k)$, $\sigma\leftarrow1$, and $\mu\leftarrow\hat{\mu}^k=\frac{(\hat{v}^k)^\top \hat{s}^k}{2n}$ (involving Eqns. \eqref{eqn_compact_linsys} and \eqref{eqn_Delta_gamma_theta_phi_psi});
        \item $z^{k+1}\leftarrow \hat{z}^k+\Delta z_c,~v^{k+1}\leftarrow \hat{v}^k+\Delta v_c,~s^{k+1}\leftarrow \hat{s}^k+\Delta s_c$;
    \end{enumerate}
    \textbf{end}\vspace*{.1cm}\hrule\vspace*{.1cm}
    \textbf{Output:} $z^{k+1}$.
\end{algorithm}

\subsection{Convergence and worst-case iteration complexity}
The analysis is carried out separately for the predictor and corrector steps. We first present the common results that apply to both.
\begin{lemma}\label{lemma_r}
    Let $(z,v,s)\in\mathcal{F}^+$ and let $(\Delta z, \Delta v,\Delta s)$ be the solution of Eqn. \eqref{eqn_Newton}. Then,
    \begin{equation}\label{eqn_r_def}
        \|\Delta v\pdot \Delta s\|\leq \frac{\sqrt{2}}{4}\|r\|^2
    \end{equation}
    where
    \begin{equation}
        r\triangleq \frac{1}{\sqrt{v\pdot s}}(\sigma\mu \mathbf{1}_{2n} -v\pdot s)        
    \end{equation}
\end{lemma}
\begin{proof}
Dividing both sides of Eqn. \eqref{eqn_Newton_c} by $\sqrt{v\pdot s}$, results in 
\begin{equation}\label{eqn_x_z_r_delta}
    \sqrt{\frac{s}{v}}\pdot \Delta v + \sqrt{\frac{v}{s}}\pdot \Delta s = r
\end{equation}
Let us denote
\begin{equation}\label{eqn_p_q_def}
p\triangleq\sqrt{\frac{s}{v}}\pdot \Delta v,~q\triangleq\sqrt{\frac{v}{s}}\pdot \Delta s    
\end{equation}
we have 
\[
p+q =r,~\Delta v \pdot\Delta s=p\pdot q,~\Delta v^\top \Delta s=p^\top q.
\]
By Remark \ref{remark_positive}, $\Delta v^\top \Delta s\geq0$, so $p^\top q\geq0$. Then, we have
\[
    \sum_{p_i q_i\geq0}p_i q_i\geq-\sum_{p_i q_i<0}p_i q_i,~i=1,\cdots,2n.
\]
We can obtain the result as follows:
\[
\footnotesize
\begin{aligned}
    &\quad\|\Delta v\pdot \Delta s\|^2=\|p\pdot q\|^2=\sum_{i=1}^{2n}(p_i q_i)^2\\
    &=\sum_{p_i q_i\geq0}(p_i q_i)^2+\sum_{p_i q_i<0}(p_i q_i)^2\\
    &\leq \left(\sum_{p_i q_i\geq0}p_i q_i\right)^2+ \left(\sum_{p_i q_i<0}p_i q_i\right)^2 \leq 2 \left(\sum_{p_i q_i\geq0}p_i q_i\right)^2\\
    &(\text{we have } 4p_i q_i=(p_i+ q_i)^2-(p_i - q_i)^2\\
    &\leq2\left(\sum_{p_iq_i\geq0}\frac{1}{4}(p_i + q_i)^2\right)^2\leq 2\left(\sum_{i}\frac{1}{4}(p_i + q_i)^2\right)^2=2\left(\frac{1}{4}\|r\|^2 \right)^2.
\end{aligned}
\]
that is, $\|\Delta v\pdot \Delta s\|\leq\frac{\sqrt{2}}{4}\|r\|^2$, which completes the proof.
\end{proof}

In Algorithm \ref{alg_PC_IPM}, both pairs $(\Delta z_p,\Delta v_p,\Delta s_p)$ and $(\Delta z_c,\Delta v_c,\Delta s_c)$ are obtained by solving Eqn. \eqref{eqn_Newton_b} and thus by Remark \ref{remark_positive} we have
\begin{equation}
\Delta v_p^\top \Delta s_p\geq0,~\Delta v_c^\top \Delta s_c\geq0.    
\end{equation}
The key of Algorithm \ref{alg_PC_IPM} is that, during the predictor step, the iterate transitions from $(z^k,v^k,s^k)\in\mathcal{N}(1/4)$ to $(\hat{z}^k,\hat{v}^k,\hat{s}^k)\in\mathcal{N}(1/2)$, and during the corrector step, it returns from $(\hat{z}^k,\hat{v}^k,\hat{s}^k)\in\mathcal{N}(1/2)$ to $(z^{k+1},v^{k+1},s^{k+1})\in\mathcal{N}(1/4)$, which will be proved in Lemmas \ref{lemma_predictor_1_2} and \ref{lemma_corrector_1_4}, respectively.

\begin{lemma}\label{lemma_predictor_1_2}
(\textbf{Analysis of Predictor Step}):
Consider Algorithm \ref{alg_PC_IPM}. If $(z^k,v^k,s^k)\in\mathcal{N}(1/4)$ and the predictor step applies the step-size
\begin{equation}\label{eqn_theta_k}
    \alpha^k= \min\left(\frac{1}{2},\sqrt{\frac{\mu^{k}}{8\|\Delta v_p \pdot \Delta s_p -\Delta \mu_{p}\mathbf{1}_{2n}\|}}\right),
\end{equation}
with $\mu^k=\frac{(v^k)^\top s^k}{2n}$ and
\begin{equation}
\Delta \mu_{p}\triangleq\frac{\Delta v_p^\top \Delta s_p}{2n}.
\end{equation}
Then, $(\hat{z}^k,\hat{v}^k,\hat{s}^k)\in\mathcal{N}(1/2)$ holds.
\end{lemma}
\begin{proof}
For any step-size $\alpha\in[0,1]$, let us define 
\[
v^k(\alpha)\triangleq v^k+\alpha\Delta v_p,~s^k(\alpha)\triangleq s^k+\alpha\Delta s_p,
\]
and 
\[
\mu^k(\alpha)\triangleq\frac{(v^k(\alpha))^\top s^k(\alpha)}{2n}.
\]
By Eqn. \eqref{eqn_Newton_c} and the choice $\sigma=0$ in the predictor step, we have
\begin{equation}\label{eqn_mu_theta_k}
\mu^k(\alpha)=(1-\alpha)\mu^k+\alpha^2\frac{(\Delta v_p)^\top \Delta s_p}{2n}=(1-\alpha)\mu^k+\alpha^2\Delta \mu_{p},    
\end{equation}
and by $\Delta v_p^\top \Delta s_p\geq0~(\mu_p\geq0),~\alpha\geq0$, we have
\begin{equation}\label{eqn_mu_k_theta_mu}
    \mu^k(\alpha)\geq(1-\alpha)\mu^k.
\end{equation}
Next,
\[
\begin{aligned}
&\quad\left\| v^k(\alpha)\pdot s^k(\alpha) - \mu^k(\alpha) \mathbf{1}_{2n}\right\| \\
&=\left\|(1-\alpha)(v^k\pdot s^k-\mu^k\mathbf{1}_{2n})+\alpha^2(\Delta v_p\pdot \Delta s_p -\Delta\mu_p \mathbf{1}_{2n})\right\|\\
&\leq(1-\alpha)\|v^k\pdot s^k-\mu^k \mathbf{1}_{2n})\|+\alpha^2\|\Delta v_p\pdot \Delta s_p -\Delta\mu_p \mathbf{1}_{2n}) \|\\
&\leq\frac{1-\alpha}{4}\mu^k +\alpha^2\|\Delta v_p\pdot \Delta s_p -\Delta\mu_p \mathbf{1}_{2n}\|,
\end{aligned}
\]
and by the choice of $\alpha^k$ in Eqn. \eqref{eqn_theta_k}, for any $\alpha\in[0,\alpha^k]$
\[
0\leq\alpha\leq\frac{1}{2} \text{ and } 8\alpha^2\|\Delta v_p\pdot \Delta s_p-\Delta \mu_p \mathbf{1}_{2n}\|\leq\mu^k,   
\]
thus we have
\[
\begin{aligned}
&\quad\left\| v^k(\alpha)\pdot s^k(\alpha) - \mu^k(\alpha) \mathbf{1}_{2n}\right\| \\
&\leq\frac{1-\alpha}{4}\mu^k +\frac{1}{8}\mu^k\leq\frac{1-\alpha}{4}\mu^k\left(1+\frac{1}{2(1-\alpha)}\right)\\
&\leq\frac{1-\alpha}{2}\mu^k\leq\frac{1}{2}\mu^k(\alpha),
\end{aligned}
\]
which can imply that for all $\alpha\in[0,\alpha^k]$: $v^k(\alpha)\pdot s^k(\alpha)\geq\frac{1}{2}\mu^k(\alpha) \mathbf{1}_{2n}$, and further by continuity: $v^k(\alpha)>0$ and $s^k(\alpha)>0$. This completes the proof of $(\hat{z}^k,\hat{x}^k,\hat{z}^k)\in\mathcal{N}(1/2)$.
\end{proof}

\begin{lemma}\label{lemma_corrector_1_4}
(\textbf{Analysis of Corrector Step}): Consider Algorithm \ref{alg_PC_IPM}. If $(\hat{z}^k,\hat{x}^k,\hat{z}^k)\in\mathcal{N}(1/2)$, then $(z^{k+1},v^{k+1},s^{k+1})\in\mathcal{N}(1/4)$. 
\end{lemma}
\begin{proof}
For any step-size $\alpha\in[0,1]$, let us define 
\[
v^{k+1}(\alpha)\triangleq \hat{v}^k+\alpha\Delta v_c,~s^{k+1}(\alpha)\triangleq \hat{s}^k+\alpha\Delta s_c,
\]
and 
\[
\mu^{k+1}(\alpha)\triangleq\frac{(v^{k+1}(\alpha))^\top s^{k+1}(\alpha)}{2n},
\]
then by Eqn. \eqref{eqn_Newton_a} and the choice $\sigma=1$ in the corrector step, we have
\begin{equation}\label{eqn_mu_theta_k_1}
\mu^{k+1}(\alpha)=\hat{\mu}^k + \alpha^2\frac{\Delta v_c^\top \Delta s_c}{2n},    
\end{equation}
where $\hat{\mu}^k=\frac{(\hat{v}^k)^\top \hat{s}^k}{2n}$. By $\Delta v_c^\top \Delta s_c\geq0$ and $\alpha\geq0$, we have
\begin{equation}\label{eqn_mu_k1}
    \mu^{k+1}(\alpha)\geq\hat{\mu}^k.
\end{equation}
Next, let us define 
\begin{equation}
    \Delta \mu_{c}\triangleq\frac{\Delta v_c^\top \Delta s_c}{2n},
\end{equation}
and we have
\[
\begin{aligned}
&\quad\left\|v^{k+1}(\alpha)\pdot s^{k+1}(\alpha)-\mu^{k+1}(\alpha)\mathbf{1}_{2n}  \right\|    \\
&=\left\|(1-\alpha)(\hat{v}^k\pdot \hat{s}^k-\hat{\mu}^k\mathbf{1}_{2n} )+\alpha^2(\Delta v_c\pdot \Delta s_c -\Delta\mu^c \mathbf{1}_{2n} )\right\|\\
&\leq(1-\alpha)\|\hat{v}^k\pdot \hat{s}^k-\hat{\mu}^ke\|+\alpha^2\|\Delta v_c\pdot \Delta s_c -\Delta\mu_c \mathbf{1}_{2n} \|\\
&\leq\frac{(1-\alpha)\hat{\mu}^k}{2}+\alpha^2\|\Delta v_c\pdot \Delta s_c -\Delta\mu_c \mathbf{1}_{2n}\|.
\end{aligned}
\]
Focusing on the term $\|\Delta v_c\pdot \Delta s_c -\Delta\mu_c \mathbf{1}_{2n}\|$,
\[
\begin{aligned}
&\quad\|\Delta v_c\pdot \Delta s_c -\Delta\mu_c \mathbf{1}_{2n} \|\\
&=\sqrt{\sum_{i=1}^{2n}(\Delta v_{c,i}\Delta s_{c,i})^2-2\Delta \mu_{c} (\Delta v_c^\top\Delta s_c) + 2n (\Delta \mu_{c} )^2}\\
&= \sqrt{\sum_{i=1}^n(\Delta  v_{c,i}\Delta s_{c,i})^2 - 2n (\Delta \mu_{c} )^2 }\\
&\leq\left\|\Delta v_c \pdot \Delta s_c \right\|
\end{aligned}
\]
By Lemma \ref{lemma_r} and the choice $\sigma=1$ in the corrector step, we have
\[
\left\|\Delta v_c \pdot \Delta s_c \right\| \leq\frac{\sqrt{2}}{4}\left\| \frac{1}{\sqrt{\hat{v}^k\pdot \hat{s}^k}}(\hat{\mu}^k\mathbf{1}_{2n} -\hat{v}^k\pdot \hat{s}^k)\right\|^2.
\]
Because $(\hat{z}^k,\hat{x}^k,\hat{z}^k)\in\mathcal{N}(1/2)$, we have $\|\hat{v}^k\pdot \hat{s}^k-\hat{\mu}^k\mathbf{1}_{2n}\|\leq\frac{1}{2}\hat{\mu}^k$, so for each $i=1,\cdots,2n$
\[
\frac{1}{2}\hat{\mu}^k\leq \hat{v}_i^k \hat{s}_i^k\leq\frac{3}{2}\hat{\mu}^k \Leftrightarrow \Big(\min(\sqrt{\hat{v}^k\pdot \hat{s}^k}\Big)^2=\frac{1}{2}\hat{\mu}^k
\]
Thus, we have
\begin{equation}\label{eqn_useful}
    \begin{aligned}
    &\quad \left\| \frac{1}{\sqrt{\hat{v}^k\pdot \hat{s}^k}}(\hat{\mu}^k\mathbf{1}_{2n} -\hat{v}^k\pdot \hat{s}^k)\right\|^2\\
    &\leq\frac{1}{(\min(\sqrt{\hat{v}^k\pdot \hat{s}^k}))^2}\|\hat{\mu}^k\mathbf{1}_{2n} -\hat{v}^k\pdot \hat{s}^k\|^2\leq\frac{1}{2}\hat{\mu}^k
    \end{aligned}
\end{equation}
Then, by Eqn. \eqref{eqn_mu_k1}, for any step-size $\alpha\in[0,1]$, we have
\[
\begin{aligned}
&\quad\|v^{k+1}(\alpha)\pdot s^{k+1}(\alpha)-\mu^{k+1}(\alpha)\mathbf{1}_{2n}\|\\
&\leq\frac{(1-\alpha)\hat{\mu}^k}{2}+\alpha^2 \left\|\Delta v_c \pdot \Delta s_c \right\|\\
&\leq\frac{(1-\alpha)\hat{\mu}^k}{2}+\alpha^2\frac{\sqrt{2}}{4}\frac{1}{2}\hat{\mu}^k\leq\frac{(1-\alpha)\hat{\mu}^k}{2}+\alpha^2\frac{2}{8}\hat{\mu}^k\\
&=\frac{(\alpha-1)^2+1}{4}\hat{\mu}^k\leq\frac{(\alpha-1)^2+1}{4}\mu^{k+1}(\alpha)\leq\frac{1}{2}\mu^{k+1}(\alpha).
\end{aligned}
\]
which can imply that for all $\alpha\in[0,1]$
\[
v^{k+1}(\alpha)\pdot s^{k+1}(\alpha)\geq\frac{1}{2}\mu^{k+1}(\alpha)\mathbf{1}_{2n}  
\]
and namely by continuity $v^{k+1}(\alpha)>0$ and $s^{k+1}(\alpha)>0$. 

Moreover, in the special case that $\alpha=1$, we have $\frac{(\alpha-1)^2+1}{4}=\frac{1}{4}$ and it proves that
\[
\|v^{k+1}\pdot s^{k+1}-\mu^{k+1}\mathbf{1}_{2n}\|\leq\frac{1}{4}\mu^{k+1},
\]
which completes the proof.
\end{proof}

\begin{lemma}\label{lemma_delta_mu}
    Consider Algorithm \ref{alg_PC_IPM}. At the $k$-th iteration, the following inequalities
    \begin{subequations}
        \begin{align}
             \Delta \mu_p&\leq\frac{1}{4}\mu^k\label{eqn_delta_mu_a}\\
             \hat{\mu}^k&\leq\left(1-\frac{\alpha^k}{2}\right)^2\mu^k\label{eqn_delta_mu_b}\\
              \Delta \mu_c&\leq\left(1-\frac{\alpha^k}{2}\right)^2\frac{1}{16n}\mu^k\label{eqn_delta_mu_c}
        \end{align}
    \end{subequations}
    hold.
\end{lemma}
\begin{proof}
Taking the $2$-norm on both sides of Eqn. \eqref{eqn_x_z_r_delta} (by the definition in Eqn. \eqref{eqn_p_q_def}) results in
\[
    4\Delta v^\top \Delta s + \|p-q\|^2 = \|r\|^2
\]
which implies that (by the definition in Eqn. \eqref{eqn_r_def})
\begin{equation}
    4\Delta v^\top \Delta s\leq\|r\|^2=\left\|\frac{1}{\sqrt{v\pdot s}}(\sigma\mu \mathbf{1}_{2n}-v\pdot s) \right\|^2
\end{equation}

Regarding the inequality \eqref{eqn_delta_mu_a}, the predictor step adopts $\sigma\leftarrow0$ and $(z,v,s)=(z^k,v^k,s^k)$, this yields
\[
\begin{aligned}
\Delta v_p^\top \Delta s_p&\leq\frac{1}{4}\left\| \frac{1}{\sqrt{v^k\pdot s^k}}(0\mu^k\mathbf{1}_{2n}-v^k\pdot s^k)\right\|^2\\
&= \frac{1}{4}(v^k)^\top s^k\Leftrightarrow\Delta \mu_p\leq\frac{1}{4}\mu^k,  
\end{aligned}
\]
which completes the proof of the inequality \eqref{eqn_delta_mu_a}.

Regarding the inequality \eqref{eqn_delta_mu_b}, by Eqn. \eqref{eqn_mu_theta_k} and  the statement \textit{i)}, we have
\[
\begin{aligned}
    \hat{\mu}^k=(1-\alpha^k)\mu^k+(\alpha^k)^2\Delta \mu_p&\leq(1-\alpha^k)\mu^k+\frac{(\alpha^k)^2}{4}\mu^k\\
    &=\left(1-\frac{\alpha^k}{2}\right)^2\mu^k,
\end{aligned}
\]
which completes the proof of the inequality \eqref{eqn_delta_mu_b}.

Regarding the inequality \eqref{eqn_delta_mu_c} the corrector step adopts $(z,v,s)=(\hat{z}^k,\hat{v}^k,\hat{s}^k)$, $\sigma\leftarrow1$, and $\mu\leftarrow\hat{\mu}^k$, and by Eqn. \eqref{eqn_useful},  this yields
\[
\Delta v_c^\top \Delta s_c\leq\frac{1}{4}\left\|\frac{1}{\sqrt{\hat{v}^k\pdot \hat{s}^k}}(\hat{\mu}^k \mathbf{1}_{2n}-\hat{v}^k\pdot \hat{s}^k) \right\|^2\leq\frac{1}{8}\hat{\mu}^k,
\]
that is (by the inequality \eqref{eqn_delta_mu_b}),
\[
\Delta\mu_c\leq\frac{1}{16n}\hat{\mu}^k\leq\left(1-\frac{\alpha^k}{2}\right)^2\frac{1}{16n}\mu^k,
\]
which completes the proof of the inequality \eqref{eqn_delta_mu_c}.
\end{proof}

\begin{theorem}
    Let $\{(z^k,v^k,s^k) \}$ be generated by Algorithm \ref{alg_PC_IPM}. Then
    \begin{equation}
        \mu^{k+1}\leq\left(1-\frac{0.2348}{\sqrt{2n}}\right)^2\mu^k
    \end{equation}

    Furthermore, Algorithm \ref{alg_PC_IPM} requires at most
    \begin{equation}\label{eqn_N_max}
        N_{\max}=\left\lceil \frac{\log(\frac{2n}{\epsilon})}{-2\log\left(1-\frac{0.2348}{\sqrt{2n}}\right)}\right\rceil.
    \end{equation}
\end{theorem}
\begin{proof}
By Eqn. \eqref{eqn_mu_theta_k} and \eqref{eqn_mu_theta_k_1}
\[
\begin{aligned}
\mu^{k+1}&=\hat{\mu}^k+\Delta\mu_c=(1-\alpha^k)\mu^k+(\alpha^k)^2\Delta\mu_p+\Delta\mu_c\\
&\quad\quad(\text{by Lemma \ref{lemma_delta_mu}})\\
&\leq(1-\alpha^k)\mu^k+\frac{(\alpha^k)^2}{4}\mu^k+\left(1-\frac{\alpha^k}{2}\right)^2\frac{1}{16n}\mu^k\\
&\leq\left(1-\frac{\alpha^k}{2}\right)^2\left(1+\frac{1}{16n}\right)\mu^k.
\end{aligned}
\]

Because we have
\[
\begin{aligned}
&\quad\|\Delta v_p\pdot \Delta s_p -\Delta\mu_p \mathbf{1}_{2n}\|\\
&=\sqrt{\sum_{i=1}^{2n}(\Delta v_{p,i}\Delta s_{p,i})^2-2\Delta \mu_{p} \Delta v_p^\top \Delta s_p+ 2n (\Delta \mu_{p} )^2}\\
&= \sqrt{\sum_{i=1}^{2n}(\Delta v_{p,i}\Delta s_{p,i})^2 - 2n (\Delta \mu_{p} )^2 }\leq\left\|\Delta v_p \pdot \Delta s_p \right\|\\
&\quad\quad\text{(By Lemma \ref{lemma_r} and the choice $\sigma=0$)}\\
&\leq \frac{\sqrt{2}}{4}\left\| \frac{1}{\sqrt{v^k\pdot s^k}}(-v^k\pdot s^k)\right\|^2=\frac{\sqrt{2}}{4}(2n)\mu^k,\\
\end{aligned}
\]
the choice of $\alpha^k$ in Eqn. \eqref{eqn_theta_k} can imply that for all $n\geq1$
\[
\alpha^k\geq\min\left(\frac{1}{2},~\frac{2^{0.25}}{2}\frac{1}{\sqrt{2n}}\right)\geq\frac{2^{0.25}}{2}\frac{1}{\sqrt{2n}},
\]
which can imply that the following condition for all $n\geq2$ ($n=1$ no need for optimization)
\[
\begin{aligned}
&\quad\left(1-\frac{\alpha^k}{2}\right)^2\left(1+\frac{1}{16n}\right)\\
&\leq\left( 1-\frac{\alpha^k}{2}\right) \left(1-\frac{2^{0.25}}{4}\frac{1}{\sqrt{2n}}+\frac{1}{16n}-\frac{2^{0.25}}{64n\sqrt{2n}}\right)\\
&\leq\left( 1-\frac{\alpha^k}{2}\right)\left(1-\frac{2^{0.25}}{4}\frac{1}{\sqrt{2n}}+\frac{1}{16n}\right)\\
&=\left( 1-\frac{\alpha^k}{2}\right)\left(1-\Big(\frac{2^{0.25}-\frac{1}{2\sqrt{2n}}}{4}\Big)\frac{1}{\sqrt{2n}}\right)\\
&\leq\left(1-\frac{2^{0.25}}{4}\frac{1}{\sqrt{2n}}\right)\left(1-\frac{2^{0.25}-0.25}{4}\frac{1}{\sqrt{2n}}\right)\\
&\leq\left(1-\frac{2^{0.25}-0.25}{4}\frac{1}{\sqrt{2n}}\right)^2=\left(1-\frac{0.2348}{\sqrt{2n}}\right)^2
\end{aligned}
\]
holds, which proves that $\mu^{k+1}\leq\left(1-\frac{0.2348}{\sqrt{2n}}\right)^2\mu^k$.

Based on the above and the initial value $\mu^0=1$,
\[
\mu^{k}\leq\left(1-\frac{0.2348}{\sqrt{2n}}\right)^{2k}\mu^0=\left(1-\frac{0.2348}{\sqrt{2n}}\right)^{2k}
\]
Hence $(v^k)^\top s_k\leq\epsilon$ holds if $2n\left(1-\frac{0.2348}{\sqrt{2n}}\right)^{2k}\leq\epsilon$.
Taking logarithms gives that $2k \log\left(1-\frac{0.2348}{\sqrt{2n}}\right)+\log(2n)\leq\log\epsilon$,
which holds if $k \geq N_{\max}=\left\lceil \frac{\log(\frac{2n}{\epsilon})}{-2\log\left(1-\frac{0.2348}{\sqrt{2n}}\right)}\right\rceil$, which completes the proof.
\end{proof}

\begin{remark}\label{remark_fast}
Algorithm \ref{alg_PC_IPM} exhibit $O(\log n)$ or $O(n^{0.25})$ iteration behavior in practice while the worst-case iteration bound $N_{\max}$ in Eqn .~\eqref{eqn_N_max} is $O(\sqrt{n})$ from the conservative estimate of the adaptive step-size $\alpha^k=\Omega(1/\sqrt{n})$ via Eqn.~\eqref{eqn_r_def} in Lemma \ref{lemma_r}. Refs.~\cite{mizuno1989anticipated} and \cite{mizuno1990anticipated} provide probabilistic proofs that, in linear programming cases, the adaptive step sizes satisfy $\alpha^k = \Omega(1/n^{0.25})$ and $\alpha^k = \Omega(1/\log n)$, respectively.
\end{remark}

\section[Numerical Examples]{Numerical Examples\footnotemark}
\footnotetext{The MATLAB code for Algorithm \ref{alg_PC_IPM} and numerical examples are publicly available at \url{https://github.com/liangwu2019/PC_BoxQP}.}

\begin{figure}
    \centering
    \includegraphics[width=1\linewidth]{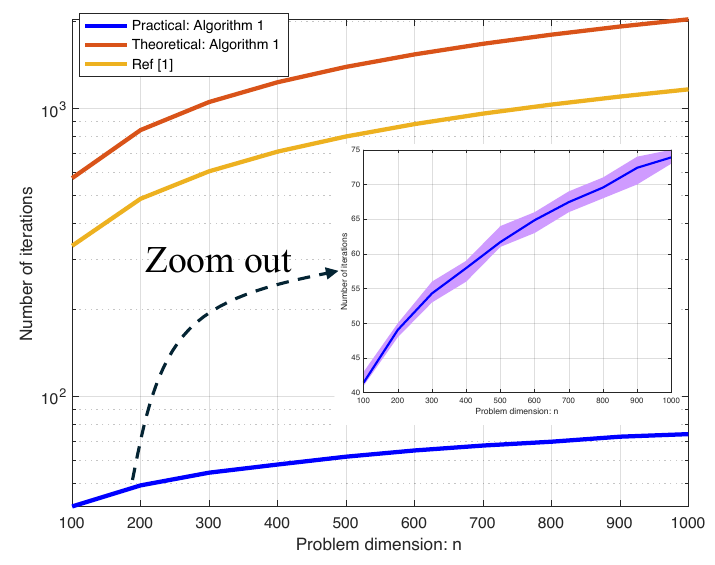}
    \caption{Practical and theoretical iteration counts of Algorithm~\ref{alg_PC_IPM} compared with the exact (practical = theoretical) iteration counts from Ref. \cite{wu2025direct}.}
    \label{fig_iters}
\end{figure}

\subsection{Practical and theoretical behavior on random Box-QPs}
We apply Algorithm~\ref{alg_PC_IPM} to random Box-QP problems with dimensions $n$ ranging from $100$ to $1000$, and evaluate its practical iteration counts. These are then compared with the theoretical worst-case bound $N_{\max}$ in Eqn.~\eqref{eqn_N_max} and with the results reported in \cite{wu2025direct}. Fig. \ref{fig_iters} shows that Algorithm \ref{alg_PC_IPM} practically behaves far less iterations than the $O(\sqrt{n})$-iteration-complexity results in Eqn.~\eqref{eqn_N_max} and Ref. \cite{wu2025direct}, which corresponding to Remark \ref{remark_fast}. Fig. \ref{fig_iters} also shows that the practical number of iterations of Algorithm \ref{alg_PC_IPM} exhibits very small variations for Box-QPs of the same dimension.

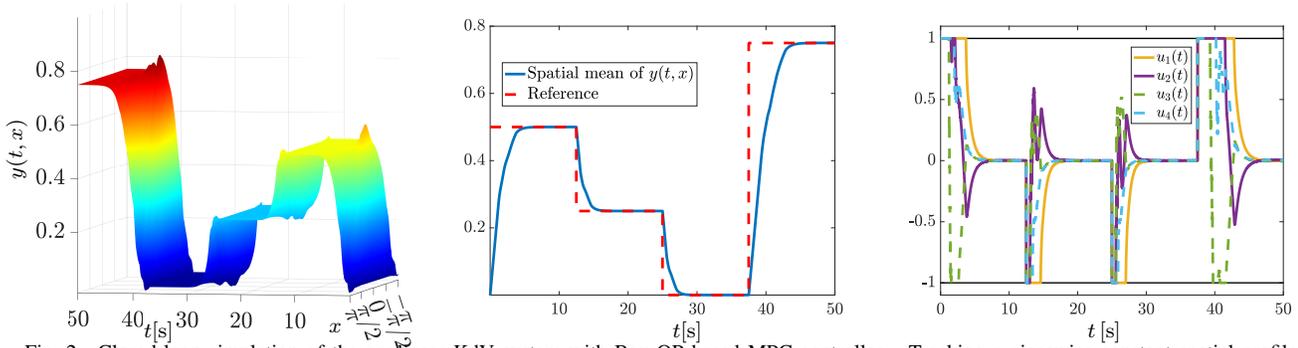
\begin{figure*}[!htbp]\label{fig}
\begin{picture}(140,110)
\put(-30,-30){\includegraphics[width=75mm]{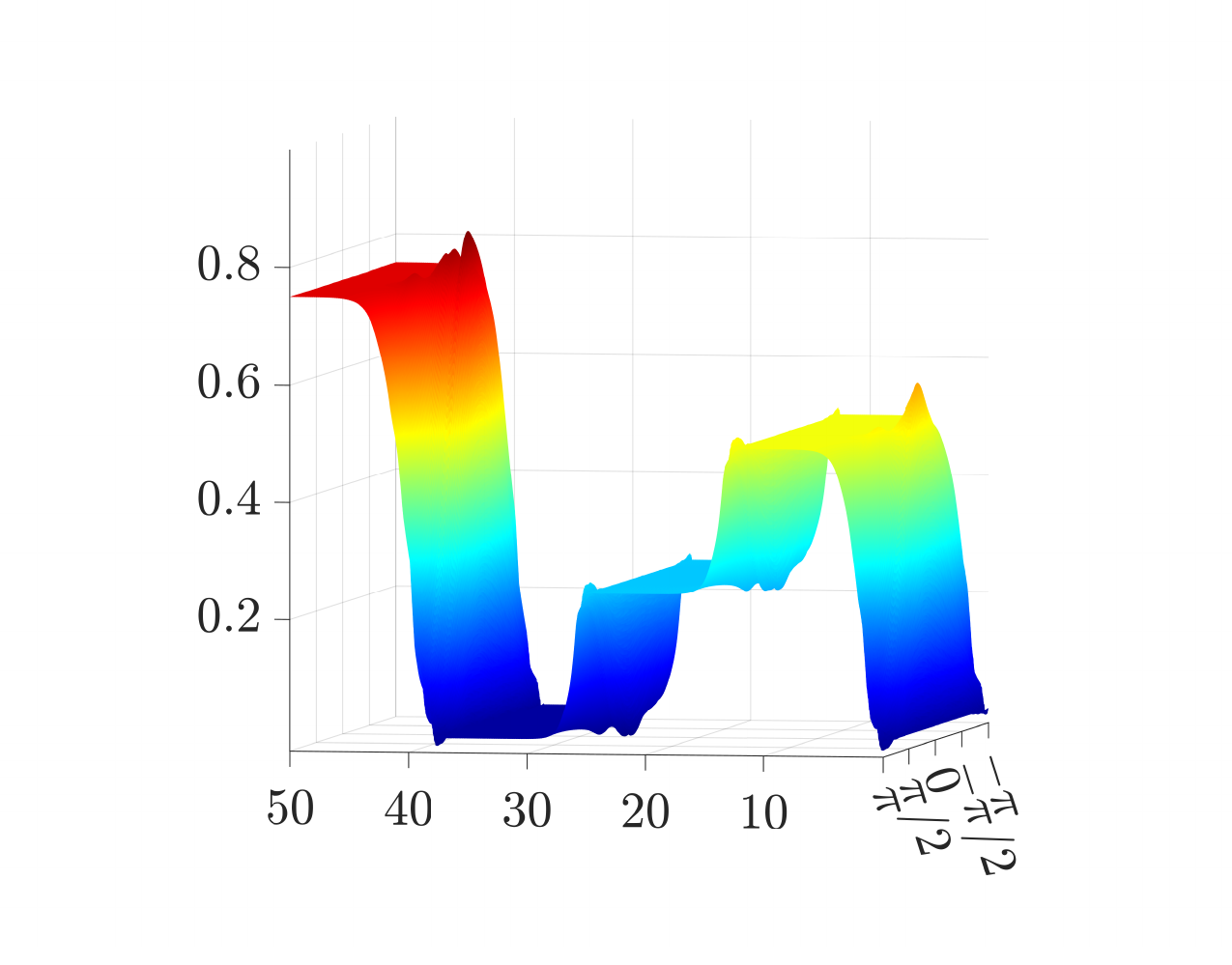}}
\put(160,-10){\includegraphics[width=55mm]{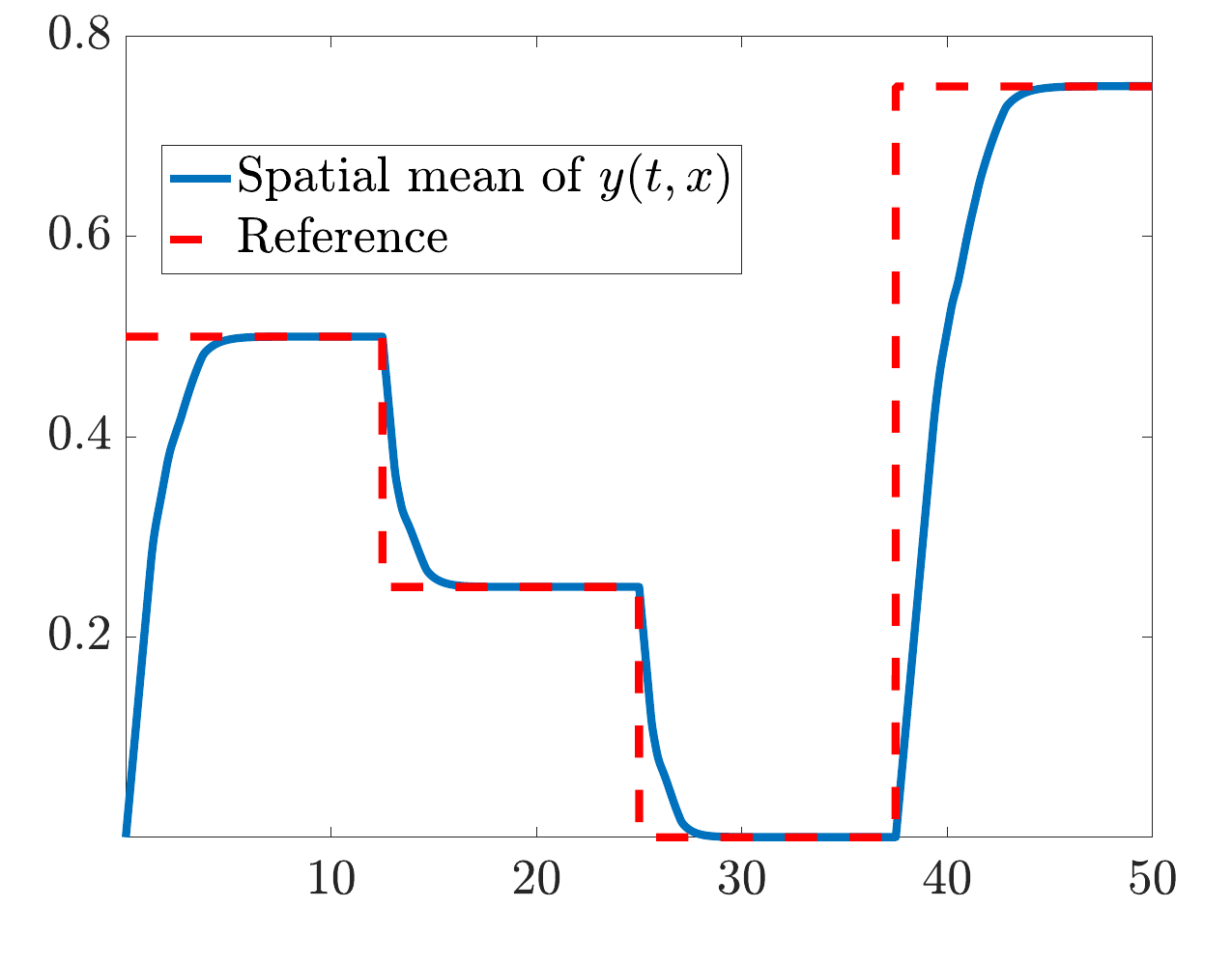}}
\put(330,-10){\includegraphics[width=55mm]{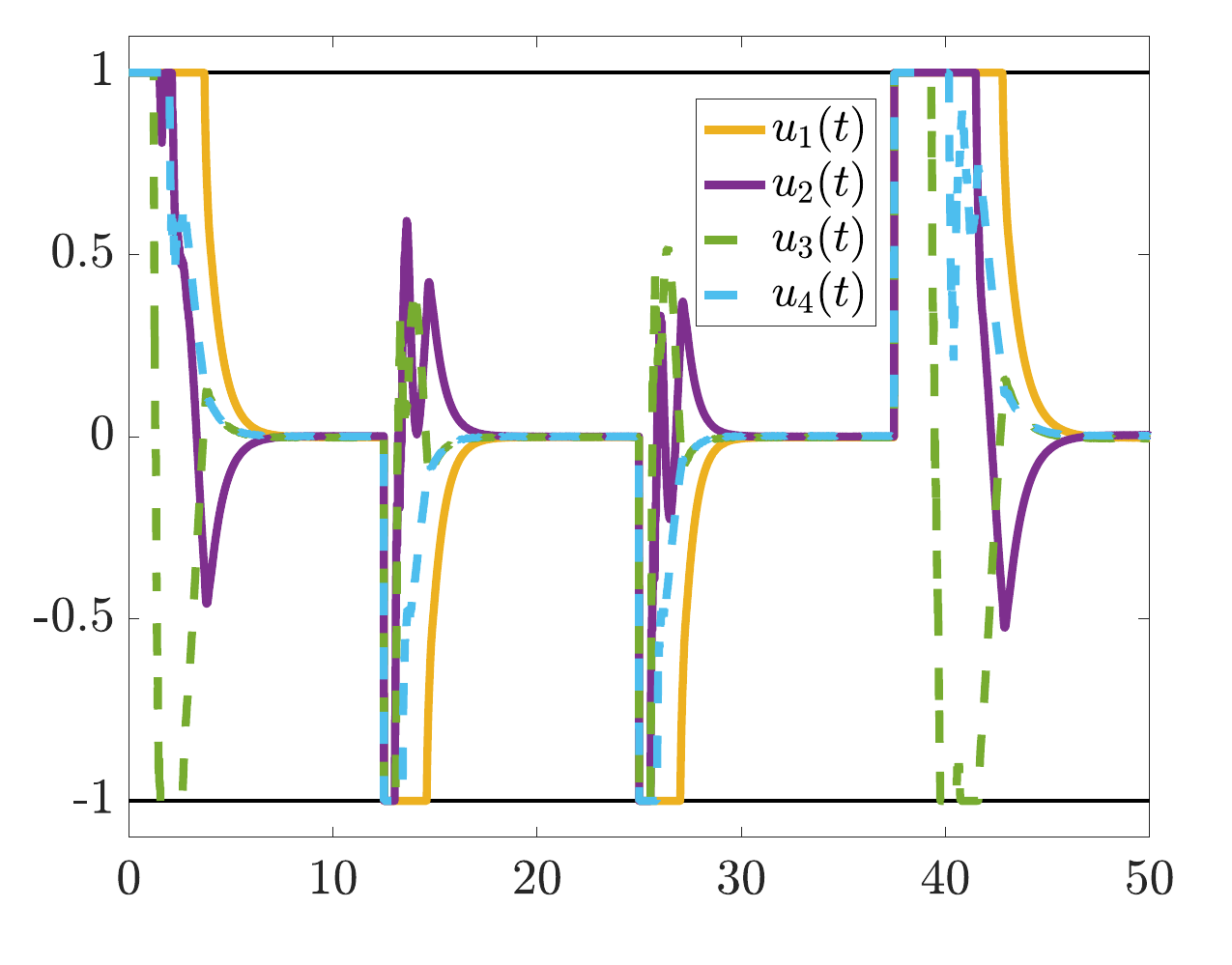}}
\put(115,-8){\footnotesize $x$}
\put(45,-10){\footnotesize $t [\mr{s}]$}
\put(-7,50){\footnotesize \rotatebox{90}{$y(t,x)$}}
\put(245,-10){\footnotesize $t [\mr{s}]$}
\put(402,-10){\footnotesize $t\,[\mr{s}]$}
\end{picture}
\caption{Closed-loop simulation of the nonlinear KdV system with Box-QP based MPC controller -- Tracking a piecewise constant spatial profile reference. Left: time evolution of the spatial profile $y(t,x)$. Middle: spatial mean of the $y(t,x)$. Right: the four control inputs.}
\label{fig_kdv}
\end{figure*}

\subsection{Nonlinear PDE-MPC case study}
We apply Algorithm \ref{alg_PC_IPM} to a nonlinear PDE-MPC problem from \cite{wu2024time}. The considered PDE plant is the nonlinear Korteweg-de Vries (KdV) equation as follows,
\begin{equation}\label{eqn_KdV}
\frac{\partial y(t, x)}{\partial t}+y(t, x) \frac{\partial y(t, x)}{\partial x}+\frac{\partial^3 y(t, x)}{\partial x^3}=u(t,x)
\end{equation}
where $x\in[-\pi,\pi]$ is the spatial variable. We consider the control input $u$ to be $u(t,x)=\sum_{i=1}^4u_i(t)v_i(x)$, in which the four coefficients $\{u_i(t)\}$ are subject to the constraint $[-1,1]$, and $v_i(x)$ are predetermined spatial profiles given as $v_i(x)=e^{-25(x-m_i)^2}$, with $m_1=-\pi/2$, $m_2=-\pi/6$, $m_3=\pi/6$, and  $m_4=\pi/2$. The control objective is to adjust ${u_i(t)}$ so that the spatial profile $y(t,x)$ tracks the given reference signal. We follow the same problem settings from \cite{wu2024time}: 1) generate data from the KdV equation \eqref{eqn_KdV}; 2) use the data-driven Koopman operator to identify a high-dimensional linear state-space model; 3) formulate the condensed MPC problem into a Box-QP \eqref{eqn_Box_QP}. The prediction horizon in \cite{wu2024time} is chosen as $10$, so the dimension of the resulting Box-QP \eqref{eqn_Box_QP} is $n=4\times 10=40$, and we adopt the stopping criteria $\epsilon=1\times10^{-6}$. Tab. \ref{tab2} shows that, compared to the algorithm proposed in Ref. \cite{wu2025direct}, Algorithm \ref{alg_PC_IPM} is $5\times$ faster, approximately half the $202/29$. Because the cost per iteration of Algorithm \ref{alg_PC_IPM} is twice that of the one in Ref. \cite{wu2025direct} (solving two different linear equations). Fig.\  \ref{fig_kdv} shows that Algorithm \ref{alg_PC_IPM} provides quick and accurate tracking of the spatial profile $y(t,x)$ to the given reference profile and zero-violations in the control input.
\begin{table}[!htbp]
\caption{Computation behvior of Algorithm \ref{alg_PC_IPM} and the method in Ref. \cite{wu2025direct} in the PDE-MPC example}
\centering
\begin{tabular}{crc}
\toprule
Methods  & Number of iterations &  Execution time [s]\footnotemark \\ \midrule
Ref \cite{wu2025direct} &  202  &  $3.4\times 10^{-3}$\\ \hline
\multirow{2}{*}{Algorithm \ref{alg_PC_IPM}} & Average: $\mathbf{29.2758\pm 2.3843}$ & \multirow{2}{*}{$\mathbf{6.6024\times 10^{-4}}$} \\ 
                         & Worst-case: 271 &  \\  
\bottomrule
\end{tabular}
\label{tab2}
\end{table}
\footnotetext{The execution time results were based on MATLAB implementations running on a Mac mini with an Apple M4 Chip (10-core CPU and 16 GB RAM). Further speedup can be achieved via C-code implementation.} 
\section{Conclusion}
This letter presents a significant improvement over our previous work \cite{wu2025direct} with a new predictor–corrector IPM algorithm, preserving the data-independent and simple-calculated $O(\sqrt{n})$-iteration-complexity, while achieving a $5\times$ speedup.

\textbf{Limitation:} Algorithm \ref{alg_PC_IPM} and the heuristic Mehrotra’s predictor–corrector IPM behave similarly in practical iteration complexity, but Algorithm \ref{alg_PC_IPM} is slower because it solves two distinct linear systems at each iteration, whereas the latter does not. Future work will address this and then perform more numerical comparisons with other state-of-the-art solvers.

\bibliographystyle{IEEEtran}
\bibliography{ref} 
\end{document}